\documentclass[11pt]{amsart} \usepackage{amsmath,amssymb, amsthm} \usepackage{pdfsync} \usepackage{color}

\newcommand\hh{\hyp^2 \times \hyp^2}  \newcommand\R{{\mathbb R}}  \newcommand\hyp{{\mathbb H}} \newcommand\C{{\mathbb C}}     \newcommand\inv{^{-1}} \newcommand\half{\frac{1}{2}}

\newtheorem{theorem}{Theorem}[section] \newtheorem{proposition}[theorem]{Proposition} \newtheorem{lemma}[theorem]{Lemma} \newtheorem{claim}[theorem]{Claim} \newtheorem{corollary}[theorem]{Corollary} \newtheorem{defn}[theorem]{Definition}    

\DeclareMathOperator{\Isom}{Isom}

\begin{document} \title{Quasi-Isometric Embeddings of Symmetric Spaces} \author{ David Fisher and Kevin Whyte}\thanks{First author partially supported by NSF grant DMS-1308291. Second author partially supported by NSF Grant DMS-1007236. The authors would also like to thank the FIM at ETHZ for hospitality and support at several points during the development of these ideas.}

\begin{abstract}
This paper opens the study of quasi-isometric embeddings of symmetric spaces.  The main focus is on the
case of equal and higher rank.  In this context some expected rigidity survives, but some surprising examples also exist. In particular there exist quasi-isometric embeddings between spaces $X$ and $Y$ where there is no isometric embedding of $X$ in $Y$.  A key ingredient in our proofs of rigidity results is a direct generalization of the Mostow-Morse Lemma in higher rank.  Typically this lemma is replaced by the {\em quasi-flat} theorem which says that maximal quasi-flat is within bounded distance of a finite union of flats.  We improve this by showing that the {\em quasi-flat} is in fact flat off of a subset of codimension $2$.
\end{abstract}

\maketitle

\section{Introduction}

The rigidity theorems of Mostow and Margulis are among the most celebrated results about the intersection of discrete groups and geometry.     With the rise of Gromov's program for the geometric study of discrete groups, coarse analogues of these results were among the most desired results \cite{Gromov}.     There are many possible translations of these theorems to a coarse setting, and so results and questions in this direction  (see \cite{Farb} for a good survey). We first recall two basic definitions:

\begin{defn} \label{defn:qi} Let $(X,d_X)$ and $(Y,d_Y)$ be metric spaces. Given real numbers $K{\geq}1$ and $C{\geq}0$,a map $f:X{\rightarrow}Y$ is called a {\em $(K,C)$-quasi-isometry} if \begin{enumerate} \item $\frac{1}{K}d_X(x_1,x_2)-C{\leq}d_Y(f(x_1),f(x_2)){\leq}K d_X(x_1,x_2)+C$ for all $x_1$ and $x_2$ in $X$, and, \item the $C$ neighborhood of $f(X)$ is all of $Y$. \end{enumerate} \noindent If $f$ satisfies $(1)$ but not $(2)$, then $f$ is called a {\em $(K,C)$-quasi-isometric embedding}. \end{defn}

\noindent{\bf Remark:} Throughout this paper, all symmetric spaces will have no compact or Euclidean factors.

To quickly summarize the current state of knowledge, good analogues of Mostow rigidity are now known in the coarse setting.    In particular, from the work of many people, one can deduce:

\begin{theorem} Two symmetric spaces are quasi-isometric if and only if they are isometric (after a possible rescaling).   Further, any finitely generated group quasi-isometric to a symmetric space  is virtually a cocompact lattice in its isometry group. \end{theorem}

See \cite{Farb} for a detailed discussion and attribution. This theorem is not quite as strong as one would like to parallel Mostow rigidity - in particular, self quasi-isometries of symmetric spaces or quasi-isometries between cocompact lattices in the same semi-simple Lie group can be quite wild, while Mostow's result says group isomorphisms are induced by isometries.     That is simply the truth for $\hyp^n$ and $\C \hyp^n$.   For the other irreducible symmetric spaces, by results of Kleiner-Leeb and Pansu one has more \cite{KL,P}:

\begin{theorem}  Let $X$ be an irreducible symmetric space of higher rank, or quaternionic or Cayley hyperbolic spaces.    Every quasi-isometry of $X$ is at bounded distance from an isometry. \end{theorem}

This is a very satisfying analogue of Mostow's rigidity results can be used to give quick proofs of Mostow's results for cocompact lattices (coarse analogues are also known for lattices which are not cocompact, but those are generally more difficult, see \cite{E}).     Margulis' super-rigidity results allow for distinct domain and range and homomorphisms rather than isomorphisms, and finding analogues of those results is an important problem in geometric group theory.     The most obvious geometric question along these lines is to ask whether quasi-isometric embeddings of one symmetric space in another must be at bounded distance of the inclusion of a totally geodesic symmetric subspace, a question raised in e.g. in \cite{Farb}.

In this paper we make the first significant progress on this question.    Perhaps the most surprising piece of the puzzle is that exotic embeddings exist even assuming the domain and range of have equal rank and rank at least $2$.

\begin{theorem}\label{firstexotic} For any $r>1$ there are quasi-isometric embeddings of $SL_{r+1}(\R)/O(r+1)$ into $Sp_{2r}(\R)/U(2r)$. \end{theorem}

\noindent
The theorem gives quasi-isometric embeddings between spaces of rank $r$ when there are no isometric embeddings in these cases, so the embeddings are definitely exotic.    See section \ref{ANmaps} for a more detailed discussion of the behavior of these embeddings.   Our construction is more general, and produces quasi-isometric embeddings in most of the cases where our rigidity results fail. The heart of the construction is

\begin{theorem}\label{ANisQI}  Let $G_1$ and $G_2$ be semi-simple Lie groups of equal real rank with Iwasawa decompositions $G_i = K_i A_i N_i$.   Every injective homomorphim of the solvable Lie group $A_1N_1$  as a subgroup of $A_2 N_2$ is a quasi-isometric embedding. \end{theorem}

\noindent This is relevant to the above as every symmetric space $K \backslash G$ is isometric to the solvable Lie group $AN$ coming from the Iwasawa decomposition $G=KAN$. We refer to quasi-isometries constructed using Iwasawa decompositions and homomorhisms as in Theorem \ref{ANisQI} as $AN$-maps.  Isometric embeddings, arising from homomorphisms from $G_1$ into $G_2$ are trivially also $AN$-maps.   As another application of Theorem \ref{ANisQI}, we have:

\begin{theorem}\label{theoremhh}  There is a quasi-isometric embedding of $\hh$ in $SL(3,\mathbb R)/O(3)$.  More generally there is a quasi-isometric embedding of the product of $n$ copies of $\hyp^2$ into $SL(n+1, \mathbb R)/O(n+1)$. \end{theorem}

\noindent It is not hard to adapt the proof of Theorem \ref{theoremhh} to produce quasi-isometric embeddings of $\prod^k_{i=1} SL(n_i+1, \R)$ into $SL(n+1,\R)$ whenever $\sum^k_{i=1} n_i =n$.  We make no attempt to be complete in our discussion of quasi-isometric embeddings in the reducible case, focusing instead on a particular construction that shows a somewhat more dramatic failure of rigidity.  By combining Theorems \ref{theoremhh} and \ref{firstexotic} with a simple {\em twist} described in Section \ref{ANmaps} we prove the following theorem:

\begin{theorem}\label{mostexotic} There exists a quasi-isometric embeddings of the product of $\hh$ into $SP(4,\R)/U(2)$ that is not at bounded distance from an $AN$-map. \end{theorem}

\noindent The behavior of these maps on flats is particularly surprising and we describe it in detail in Section \ref{ANmaps}. It is probably true that the same type of argument yields quasi-isometries that are not $AN$ maps from $(\hyp^2)^r$ to $Sp(4r,\R)/U(2r)$ but proving this by our method would be involve some quite complicated combinatorics.  While this paper was being revised, Nguyen produced examples of quasi-isometric embeddings of $(\hyp^2)^n$ into $SL(n+1, \mathbb C)$ where every flat maps over more than a single flat, and so the map is transparently not at bounded distance from an $AN$ map \cite{Nguyen}.

%Using Corollary \ref{corollary:classification}, we can rule out mimicking these constructions with irreducible symmetric %spaces of rank at least $2$.

In the other direction, despite the constructions in Theorems \ref{firstexotic}, \ref{ANisQI}, \ref{theoremhh} and \ref{mostexotic}, there is a substantial amount of rigidity for quasi-isometric embeddings in equal and higher rank.  We prove that quite often one has only isometric embeddings.    The key issue seems to be linear embeddings of the patterns of hyperplanes in the restricted root system.  We refer to this pattern of hyperplanes as the {\em Weyl pattern} of the symmetric space.  The reason to believe that rigidity should occur in equal and higher rank is the {\em quasi-flat theorem} of Eskin-Farb and Kleiner-Leeb which shows that the image of a flat in this setting is within a bounded distance of a finite union of flats.  To prove our rigidity results, we require a substantial improvement to this statement that we describe immediately after stating our main rigidity result.

\begin{theorem}\label{rigidityThm} Let $X_1$ and $X_2$ be irreducible symmetric spaces or Euclidean buildings, both of rank $r>1$.   Let $C_1$ and $C_2$ be their Weyl patterns on $\R^r$.  Then: \begin{enumerate}

\item If there are no elements of $GL_r(\R)$ embedding $C_1$ into $C_2$ then there are no quasi-isometric embeddings
    of $X_1$ into $X_2$.

\item Fix $K>1$.  If all embeddings of $C_1$ into $C_2$ which are $K$-quasi-conformal are conformal then all
    $(K,C)$-quasi-isometric embeddings of $X_1$ into $X_2$ are at bounded distance from a totally geodesic
    embedding.    In particular, if all pattern embeddings are conformal then all quasi-isometric embeddings are at
    bounded distance from totally geodesic.

\end{enumerate} \end{theorem}

{\noindent {\bf Remark:} While part $(1)$ of the conclusion of Theorem \ref{rigidityThm} follows easily for reducible symmetric spaces and buildings, the assumption of part $(2)$ never holds in the reducible case, not even in the context of Corollary \ref{local}.  Results in the reducible case have recently been obtained by Nguyen \cite{Nguyen}.}

To explain the key technical mechanism in the proof, our version of the Mostow-Morse Lemma, we pass to asymptotic cones and bilipschitz maps.
We do so simply to have cleaner statements of results.  We are now consider a building $Y$ of rank $k$ and in this context Kleiner and Leeb showed that any bilipschitz map $f \R^k \rightarrow Y$ has image contained in a finite union of flats \cite{KL}. One can also deduce this by taking asymptotic cones of a theorem of Eskin and Farb \cite{EF}.  Given such a map $f$, we call a point $x \in \R^k$ {\em flat} if there exists $r>0$ such that $B(x,r) \subset \R^k$ maps to a single flat in $Y$.  Given $f$, we will call the set of non-flat points for $f$ the exceptional set for $f$.

\begin{lemma}[Higher rank Mostow-Morse]
\label{lemma:morse}
Let $Y$ be an affine building of real rank $r$ and $f:\R^k \rightarrow Y$ a bilipschitz map.  Then the exceptional set $Z$ for $f$ has codimension $2$.  In fact $Z$ is contained in a finite collection of bilipschitz images linear spaces of dimension $d-2$.
\end{lemma}

\noindent Since codimension 2 sets in one dimensional flats are empty, this statement is a direct generalization of the standard Mostow-Morse lemma in hyperbolic spaces. This Lemma is proven in subsection \ref{quasiflats} and is a necessary component of our proofs. While we do not investigate this here, it seems that Lemma \ref{lemma:morse} should be true in a much broader context, namely any setting where there is a) a quasi-flat theorem and b) some combinatorial control on intersections of flats, including perhaps settings like \cite{BHS, H}. We prove the lemma from the quasi-flats theorem, it seems quite difficult to prove directly without first proving that theorem.  There is a coarse analogue of Lemma \ref{lemma:morse} for quasi-flats in symmetric spaces and buildings, but the coarse version of the statement becomes quite cumbersome.

We use Lemma \ref{lemma:morse} in proving both parts of Theorem \ref{rigidityThm}, though a proof of the first part is possible, though much more cumbersome, without it.  To prove both parts of the theorem we pass to asymptotic cones and study the map restricted to flats.  To prove the first part of the theorem, we look at a flat point $x$ of a flat $F$ where we have a bilischitz map of domains in $\R^r$.  By Rademacher-Stepanov this map is differentiable almost everywhere and the goal is to see that the derivative is a linear embedding of patterns.  To check this, one uses that the pattern arises as the intersection of $F$ with transverse flats to control the derivative.  To prove this one uses that these intersections are codimension one in both flats and so are not contained in the exceptional sets.  The proof of the second conclusion is more difficult.  The key step is to show that a flat maps to a single flat.  The hypothesis allows us to see that the derivative
along a flat is everywhere conformal and so the map on the flat is smooth off the exceptional set by a result of Gehring. The hypotheses imply that
the isometric part of the derivative is an element of the Weyl group at every point and the fact that the exceptional set is codimension 2 allows us to conclude this Weyl group element is constant and then check that the rescaling must be the same at every point.  This shows the map on the flat is globally isometric.

The use of Lemma \ref{lemma:morse} leads to a short elegant proof in our context and also allows for significant simplification of the proofs of \cite{E,D}. To conclude the proof of Theorem \ref{rigidityThm} in the setting of symmetric spaces requires an additional input, namely the main result of \cite{KL2} which we use to conclude that the image
of the quasi-isometric embedding, which is a union of isometrically embedded flats, is actually a sub-symmetric space.

We now state some corollaries of Theorem \ref{rigidityThm}.  In section \ref{patterns} we classify all pattern embeddings (note that $B_n$, $C_n$, and $BC_n$ all have the same pattern, which we will call $BC_n$).      The conclusion is that these are all conformal except for some exceptional embeddings of $A_n$ into $BC_n$ or $F_4$ and some exceptional embeddings of $A_2$ and $BC_2$ into $G_2$.

\begin{corollary}\label{corollary:classification} Among equal rank, irreducible buildings or symmetric spaces all quasi-isometric embeddings are at bound distance from totally geodesic unless the domain is of type $A_n$ and the range of type $B_n$, $C_n, BC_n$ or $F_4$ or unless the source is type $D_4$ and the range is type $A_4$ or unless the range is of type $G_2$. \end{corollary}

As simple example of this is, for example, that all quasi-isometric embeddings of $SL_n(\R)$ into $SL_n(\C)$ (for $n>2$) are near isometric embeddings.     As discussed in sections \ref{ANmaps} and \ref{patterns} our rigidity results and construction are close to complementary.    See section \ref{Open} for a discussion of the few cases where neither applies.  We also have an additional corollary concerning quasi-isometric embeddings with ``small constants".

\begin{corollary} \label{local} Let $X$ and $Y$ be irreducible buildings or symmetric spaces of equal rank $r>1$.  Then there exists a constant $K$, depending only on the Weyl patterns of $X$ and $Y$ such that for any $K'<K$ and any $C$, any $(K,C)$-quasi-isometric embedding of $X$ into $Y$ is at bounded distance from an isometry. \end{corollary}

\noindent{\bf Remark:} One might expect the gap in the Corollary to be a result of a limiting argument, taking $(K_i,C)$-quasi-isometric embeddings for some sequence $K_i \rightarrow 1$ and then analysing the case of $(1,C)$ embeddings.  That $(1,C)$ embeddings are necessarily isometric can be proven using a somewhat simplified version of the proof of Theorem \ref{rigidityThm}.  This limiting argument would yield only some non-explicit $K$ slightly larger than $1$. If one proves the Corollary directly from Theorem \ref{rigidityThm} instead, one obtains an explicitly computable $K$ that depends only on the Weyl patterns of $X$ and $Y$.

The ideas used in this paper find further application in work of the first author and Nguyen on quasi-isometric embeddings of non-uniform higher rank lattices \cite{FisherNguyen}.

We end the introduction with some remarks on the case of embeddings between rank one spaces and embeddings where the rank is allowed to increase. Without the assumptions of equal rank and higher rank the picture becomes quite opaque and it is not clear that any rigidity remains.    In rank one, a theorem of Bonk and Schramm  shows, in particular, that every rank one symmetric space quasi-isometrically embeds in $\hyp^n$ for $n$ sufficiently large (and hence into $\C\hyp^n$ and $\hyp\hyp^n$) \cite{BS}.    Thus quasi-isometric embeddings exist in many setting where there are no isometric embeddings. One can iterate these  constructions and combine them with isometric embeddings to build various exotic quasi-isometric embeddings between say $\hyp\hyp^n$ and $\hyp\hyp^m$ for some $m>>n$.   Further, even when isometric embeddings do exist, quasi-isometric embeddings can be quite far from isometric,  even equivariantly, for example quasi-fuchsian groups in $\hyp^3$ and more general bendings of cocompact lattices in $\mathbb H^n$ inside of $\Isom(\mathbb H^{n+1})$ are quasi-isometric as along as they remain convex cocompact.

When the rank of the target is higher than that of the domain things are also not well controlled, even if both are higher rank (since rank increases under quasi-isometric embeddings, there are none where the rank of the domain exceeds that of the target).     For example, any $X$ can be quasi-isometrically embedded in  $X \times \R$ as the graph of any Lipschitz function $X \to \R$.   In particular, there are  many strange embeddings of $SL(n,\R)$ into $SL(n+1,\R)$ of this type as  $SL(n,\R) \times \R$ is a totally geodesic (and full rank) subspace of $SL(n+1,\R)$.

In both of these cases it still seems to be a difficult question to determine precisely when quasi-isometric embeddings exist, see Section \ref{Open} for some open problems along these lines.

\section{$AN$-maps}\label{ANmaps}

In this section we describe a construction of some ``exotic" quasi-isometric embeddings.    These examples are built from injective homomorphisms of solvable Lie groups.    Let $G$ be a semi-simple Lie group, with Iwasawa decomposition $G = KAN$.     The solvable group $AN$ acts simply transitively on $X=K\backslash G$ by isometries, and so is quasi-isometric to $X$.    The key observation in our construction is:

\begin{proposition}\ref{ANisQI} Let $A = \R^k$ for some $k\geq 1$, and let $N_1$ and $N_2$ be two nilpotent Lie groups with $A$ actions so that $N_i$ is uniformly exponentially distorted in $AN_i$.    If $ f: N_1 \to N_2$ is an $A$-equivariant embedding of Lie groups then the induced embedding $AN_1 \to AN_2$ is a quasi-isometric embedding. \end{proposition}

\noindent{\bf Remark:} We call maps satisfying the hyphotheses of the proposition {\em $AN$ maps}.

\begin{proof}

We use the notation $\|g\|_G$ throughout for the distance in $G$ from $g$ to the identity. Here uniformly exponentially distorted means that there is a $C>0$ so that for all $n \in N$ we have $||n||_{AN} \leq C \log(1+||n||_N)$.    This holds for the solvable groups $AN$ coming from semi-simple Lie groups as $N$ is spanned by non-trivial root spaces.   It holds in many other, but not all, solvable Lie groups.

Let $\lambda$ be such that $f$ is $\lambda$-Lipschitz (such a constant exists as the map $AN_1$ to $AN_2$ is a homomorphism).  The content of the proposition is that there is a linear lower bound on distances in $AN_2$ in terms of the distances in $AN_1$.  Since $f$ is a homomorphism, it suffices to find such a bound for distance from the identity.

The following holds because the exponential map for nilpotent groups is polynomial.

\begin{claim} Let $N$ and $N'$ be simply connected nilpotent Lie groups.   For any embedding $N$ as a closed subgroup of $N'$ there is a polynomial $P$ so that for all $n \in N$  we have $||n||_{N}  \leq P(||n||_{N'})$. \end{claim}

We now return to the proof of the proposition.   Let $g=an$ be an element of $AN$  as above.    The triangle inequality gives $||g|| \leq ||n|| + ||a||$.   Note that all norms appearing in this argument are norms in the group $AN$, not in any subgroup.  Since the projection to $A$ is distance decreasing, $||g|| \geq ||a||$.      We also have: $$||g|| = d(an,Id) \geq d(an,a) - d(a,Id) = ||n||-||a||$$

Thus, if $||a|| \leq \half ||n||$, we have $||g|| \geq \half ||n||$.   Otherwise, $||g|| \geq ||a||$.   So we have:

$$||a||+||n|| \geq ||g|| \geq \half \hbox{max}(||a||,||n||) \geq \frac{1}{4}(||a||+||n||)$$

Thus the distance to the identity is bilipschitz equivalent to the distance in the product $A \times N$ where $A$ and $N$ are given the norms induced by their inclusions in $AN$.    Since $f$ is a linear isomorphism along $A$ is it bilipschitz there.     For $f$ embedding $N_1$ into $N_2$ we have that the norm in $N_2$ is bounded by a polynomial in the norm in $N_1$ by the claim, and so, by uniform exponential distortion, the norm of an $n$ in $AN_2$ is bounded by the logarithm of a polynomial of an exponential of the norm in $AN_1$.    This yields the desired linear lower bound.

\end{proof}

For $G=KAN$ we have that $N$ splits as a sum of positive root spaces $N = \oplus_{\lambda} E_{\lambda}$ where $\lambda$ is a linear function on $A$ describing the action of $A$ on the part of the Lie algebra of $N$ tangent to $E_{\lambda}$.    An $A$-equivariant embedding of $N_1$ into $N_2$ must send these root spaces into root spaces.

If $\lambda_1$ and $\lambda_2$ are positive roots in $G_1$ then either $\lambda_1 + \lambda_2 = \lambda$ for a root $\lambda$, in which case $[E_{\lambda_1},E_{\lambda_2}] = E_{\lambda}$ or $\lambda_1 + \lambda_2$ is not a root, in which case $E_{\lambda_1}$ and $E_{\lambda_2}$ commute.       Since a homomorphism respects brackets, it follows that if $E_{\lambda_i}$ maps into the $E_{\eta_i}$ and $\lambda_1 + \lambda_2 = \lambda$ then $\eta_1 + \eta_2 = \eta$ for a root $\eta$ and $E_{\lambda}$ maps into the $\eta$-root space of $N_2$.

When $\lambda_1 + \lambda_2$ is not a root in $G_1$, we must have that the images in the $\eta_1$ and $\eta_2$ root spaces in $G_2$ that commute.   For simplicity, assume that $G_2$ is $\R$-split or complex.  Then this implies that $\eta_1 + \eta_2$ is not a root in $G_2$.      Under this assumption, we have an additive map $\lambda \mapsto \eta$ sending positive roots for $G_1$ to positive roots for $G_2$ such that two roots in the domain sum to a root iff they do in the range.    This linear map is nothing but the map induced on $A$.    Conversely, given such a map, one clearly has an embedding of solvable groups.   In summary:

\begin{proposition}  Let $G_1$ and $G_2$ be semisimple real Lie groups with restricted root systems $R_1$ and $R_2$ of equal rank.     There is an embedding of $AN_1$ into $AN_2$ if there is a linear automorphism $T$ of $A$ carrying positive roots into positive root and such that $\lambda_1$ and $\lambda_2$ sum to a root iff $T(\lambda_1)$ and $T(\lambda_2)$ do.   If  the $G_i$ are  both $\R$-split, both complex, or $G_1$ split and $G_2$ complex then the converse holds . \end{proposition}

\begin{proposition} There is a quasi-isometric embedding $Sl_{n+1}(\R)$ into $Sp_{2n}(\R)$. \end{proposition}

\begin{proof} By the previous proposition, we must find a linear map $\R^n \to \R^n$ which sends the positive roots of $A_n$ to positive roots of $C_n$ respecting when roots sum to roots.   Represent the root systems as:\\

\noindent $A_n$ with positive roots $x_i -x_j$ for $i>j$ on the space\\ $\R^n = \{(x_0,\cdots,x_n): x_0 + \cdots x_n =0\}$\\

\noindent and\\

\noindent $C_n$ with positive roots $2y_i$, $y_i + y_j$, and $y_i - y_j$ with $i>j$ on\\ $\R^n = \{(y_1,\cdots,y_n)\}$\\

\noindent then the map:

$$T(x_0,\cdots x_n) = \half ( 2x_1 - (x_0 + x_n), 2x_2 - (x_0 +x_n), \cdots, 2x_{n-1}-(x_0+x_n), x_n - x_0)$$

\noindent satisfies the requirements.    The corresponding root for $x_i - x_j$ for $i>j>0$  is $y_i - y_j$, for $x_i - x_0$ for $0<i<n$ is $y_i + y_n$ and for $x_n - x_0$ is $2y_n$.

Some additional work is required to show that this matching of roots spaces actually gives a homomorphism of nilpotent groups.  Essentially given two roots $\alpha$ and $\beta$ which sum to a root $\alpha + \beta$, one has that corresponding root spaces satisfy $[X_\alpha, X_\beta] = c_{\alpha,\beta}X_{\alpha +\beta}$ and one needs to check that it is possible to chose the $c_{\alpha,\beta}$ consistently over all choices of $\alpha$ and $\beta$.  For this particular case, one can could construct an explicit matrix embedding to check this condition.  For a more theoretical discussion see subsection \ref{nonSplitAN} and particularly Corollary \ref{typeA}. \end{proof}

\noindent {\bf Remark:} This same map gives a quasi-isometric embedding $Sl_{n+1}(\C)$ into $Sp_{2n}(\C)$, or any other $A_n$ to $C_n$ situation where the domain is $\R$-split (or complex if the target is as well).  The exact same construction works over $p$-adic fields.   It is interesting to compare this to the linear pattern embeddings classified in section \ref{patterns}.

{\noindent {\bf Remark:} There is some earlier work using $AN$ maps to build quasi-isometric embeddings, but only in cases where $A$ is one dimensional.  In particular, Brady and Farb use a similar construction to show that $\mathbb H^3$ quasi-isometrically embeds in $\mathbb H^2 \times \mathbb H^2$ \cite{BF}. In rank one, this construction is quite flexible and has other applications \cite{BF,Fo,L}, but it remains quite surprising that such a construction is possible in the more combinatorially complex higher rank case.}

It is interesting to look at the geometry of this map in more detail.    Restricted to a flat in $Sl_{n+1}$ corresponding to a coset of $A$ in $AN$ the map $f$ is precisely the linear map $T$ above.      The hyperplanes in $Sp_{2n}$ that do not come from those in $Sl_{n+1}$ come from the roots $2y_i$ for $i<n$ and $y_i + y_j$ for $i<j<n$.    These correspond to the functions $2x_i = x_0 + x_n$ (for $0 <i <n$) and $x_i + x_j = x_0 + x_n$ (with $0<i<j<n$).   The chambers of the flat in the domain are given by the ordering of the coordinates, and how these missing hyperplanes subdivide such a chamber depends on where $x_0$ and $x_n$ sit in this ordering.

A flat in $A_n$ has $(n+1)!$ chambers and one in $C_n$ has $2^n n!$.   Thus the``average" chamber must map across $\frac{2^n}{n+1}$ chambers.    The minimally subdivided chambers  are those where $x_0$ and $x_n$ are either the two smallest or two largest in the ordering (these come in four families from the ordering between $x_0$ and $x_n$ and whether they are minimal or maximal) and map to single chambers in $C_n$.  Those maximally subdivided are those with $x_0$ and $x_n$ the minimum and maximum (in either order).

For concreteness we look at $n=2$.     There are six chambers in our flat, corresponding to the permutations of $\{0,1,2\}$.    There are two chambers in which $x_0$ and $x_2$ are the minimum and maximum.     The ordering $x_0 < x_1 < x_2$ is the unique chamber fixed by $N$ (call this chamber $C_0$) and the ordering $x_2 < x_1 < x_0$ is the opposite chamber.     Both of these are sent to two chambers in $Sp_4/U(2)$.   The four orderings in which $x_1$ is either minimal or maximal are all sent to single chambers.

This gives a picture of what happens for an arbitrary flat.    Divide the flat into its six chamber, and label each according to whether it is the chamber fixed by $N$ (there is just one such), adjacent to this chamber, opposite to this chamber (the generic case), or adjacent to an opposite chamber.     What happens to every chamber is then determined as above.  It is not hard to see that the flats in $Sl_3$ are of the following types:

\begin{itemize} \item A flat passing through $C_0$.   These are described above, and have two chambers (opposite to each other in
    the flat) that map to two chambers, and the four others which map to single chambers.   Thus the flat maps to a
    total of eight chambers which is necessarily a single flat.   The map on flats if given by the linear map $T$
    above.
\item There are two chambers of $F$ (next to each other  in $F$) which are opposite to $C_0$,their neighbors in $F$
    which are adjacent to opposites to $C_0$, and two chambers adjacent to $C_0$.    This maps across eight total
    chambers, and hence has image at bounded distance from a single flat.   Unlike the previous case, however, the
    map is not globally linear as a map $\R^2 \to \R^2$.
\item There are two chambers of $F$ (next to each other  in $F$) which are both adjacent to $C_0$ or both adjacent
    to opposite to $C_0$.   The remaining four chambers are all opposite to $C_0$.    Such a flat maps to an image
    quasi-flat with ten chambers, and hence is not near a single flat.
\item (The generic case) All chambers in $F$ are opposite to $C_0$.    In this case all six chambers map across two
    chambers in the range.    The image quasi-flat has twelve chambers and so in not near a single flat.
\end{itemize}

It might be interesting to work out precisely what these quasi-flats look like.   In particular, in analogy with Proposition \ref{quasiflats}, what does the locus of points where the image essentially bends (is not near a single flat) look like?  The existence of these coherent families of quasi-flats not close to flats is genuinely surprising.

We now prove Theorem \ref{theoremhh}, which we can do by providing an explicit map on matrices.  The Iwasawa decomposition of $(\hyp^2)^n$ can be realized as $AN = \prod A_iN_i$ where each $A_i$ is a two by two diagonal matrix

$$\begin{bmatrix} a_i & 0 \\ 0 & a_i{\inv} \\ \end{bmatrix}$$

{\noindent} and each $N_i$ is a two by two diagonal matrix

$$\begin{bmatrix} 1 & u_i \\ 0 & 1 \\ \end{bmatrix}$$

where each $a_i$ is a non-zero real number and each $u_i$ is a real number.

The desired $AN$ embedding is then explicitly:

$$\begin{bmatrix} \alpha_1 &  u_1 & u_2 & \ldots & u_n\\ 0 & \alpha_2 & 0 & \ldots & 0\\ \vdots & \vdots & \vdots & \vdots & \vdots &\\ 0 & 0 & \ldots & 0 & \alpha_1{\inv}\ldots\alpha_{n}{\inv}\\

\end{bmatrix}$$

\noindent where the $\alpha_i$ are easily computed explicitly.  For instance in the case $n=2$, we have $\alpha_1 = {a_1^2}{a_2^2}$ and $\alpha_2=a_2^2$.

We can proceed to the proof of Theorem \ref{mostexotic}.  First note that in an $AN$-map there is always one chamber at infinity, the one fixed by $AN$ such that any every flat passing through it maps lineary to a flat in the range.  We show a map is not an $AN$ map by showing this does not occur and in fact that no flat maps linearly from domain to range.    We begin by describing the image of flats in $\hyp^2 \times \hyp^2$ into $SL(3, \mathbb R)/O(3)$ and then combining this with the description given above for flats $SL(3, \mathbb R)/O(3)$ mapping into $SP(4, \mathbb R)/U(2)$.

   For this description, we will think of $\hyp^2$ as being metrics on $\R^2$ up to scalars, and
   ${SL_3(\R)}/{SO_3(\R)}$ as the same on $\R^3$.    The embedding is then: choose to planes $P_1$ and $P_2$ in
   $\R^3$, and identify each with $\R^2$.   A point of $\hh$ can then be thought of a a metric (up to scalar) on each
   of $P_1$ and $P_2$.   Rescale so they agree on their line of intersection (let's call that $l$), and extend that to
   an inner-product on $\R^3$ by having $P_1$ and $P_2$ meet orthogonally.

Note that this tells us the image of the embedding, namely all metrics on $\R^3$ for which $P_1$ and $P_2$ meet orthogonally.     Also note that there is a Lipschitz inverse, mapping ${SL_3(\R)}/{SO_3(\R)}$ to $\hh$ just by restricting the metric to $P_1$ and $P_2$.

The geodesics in $\hyp^2$ correspond to the set of metrics keeping a pair of lines orthogonal, and similarly a maximal flat in ${SL_3(\R)}/{SO_3(\R)}$ corresponds to the set of metrics keeping a triple of lines orthogonal.      The construction includes a distinguished line $l = P_1 \cap P_2$.   Choosing another line $l_1 \subset P_1$ and $l_2 \subset P_2$, we get a product of two geodesics in $\hh$ which naturally maps to the flat in ${SL_3(\R)}/{SO_3(\R)}$ in which $l, l_1, l_2$ are orthogonal.

 This flat in $\hh$ is subdivided into four quadrants - if we choose vectors $v \in l$, $u_1 \in l_1$, $u_2 \in l_2$
 then we can think of the quadrants as determined by which of $v$ and $v_i$ is longer in the metrics on $P_i$.
 The flat in ${SL_3(\R)}/{SO_3(\R)}$ is similarly divided into six chambers according to the ordering of sizes of
 $v$,
 $v_1$, and $v_2$.     From this point of view we can see how some chambers in the domain map to two chambers in the
 range and others to only one.     For example, if we know that $v$ is shorter than $v_1$ in $P_1$ and $v$ is longer
 than $v_2$ in $P_2$ then, once we rescale to get the metrics to agree on $v$, we know that the order is $||v_2|| <
 ||v|| < ||v_1||$ and we have a quadrant that maps to a single chamber.    On the other hand, if $v$ is shorter than
 $v_1$ in $P_1$ and shorter than $v_2$ in $P_2$, after rescaling we might have $||v||<||v_1||<||v_2||$ or
 $||v||<||v_2||<||v_1||$, so this quadrant gets mapped across two chambers.

 Viewing the chambers in $\hh$ as a points in $\partial \hyp^2 \times \partial \hyp^2$, what we see is that a chamber
 in which neither point is $l$ is mapped to two chambers, as the one chamber with points points $l$ while those in
 which one is $l$ and one is not are mapped to a single chamber.    Thinking of chambers in ${SL_3(\R)}/{SO_3(\R)}$
 as flags $V_1 \subset V_2 \subset \R^3$ and the points at infinity in $\hyp^2$ as lines in $\R^2$ what we have is
 this:

\begin{itemize} \item A chamber in $\hh$ with points at infinity $l_1$ and $l_2$, neither equal to $l$ maps to the two chambers $
    l_1 \subset V$ and $l_2 \subset V$ where $V$ is the span of $l_1$ and $l_2$.
\item A chamber in $\hh$ with points at infinity $l_1$ and $l$, $l_1 \neq l$ maps to the chamber $ l_1 \subset
    P_1$.
    Likewise, a chamber with endpoints $l$ and $l_2$ maps to the one corresponding to the flag $l_2 \subset P_2$.
\item The chamber with both points at infinity equal to $l$ maps to the two chambers $l \subset P_1$ and $l \subset
    P_2$.
\end{itemize}

From this it is easy to determine what happens to flats : if none of the endpoints is $l$ the image is a quasi-flat made of eight chambers, if both geodesics have $l$ as an endpoint the image is a single flat (and the map is linear here).    Finally if one of the geodesics has $l$  as an endpoint and the other does not the image is again a single flat ( the metrics for which the three endpoints that are not $l$ are orthogonal ) but the map is not linear - one half of the flat maps to four of the chambers and the other half to two.

The embedding of ${SL_3(\R)}/{SO_3(\R)}$ into $Sp_4 (\R)/ U(2)$ can be described similarly (by extending a metric on $\R^3$ to one on $\R^4$ compatible with the standard symplectic form) but is more complicated.  The summary for chambers follows as described above.   In this terminology, there is a distinguished flag $V_1 \subset V_2$ in $\R^3$, and this flag as well as those opposite to it (meaning those $U_1 \subset U_2$ for which $U_1 \not\subset V_2$ and $V_1 \not\subset U_2$) map to two chambers.   Flags neither equal to nor opposite to $V_1 \subset V_2$  map to a single chamber.

We can compose these constructions to get an exotic embedding of $\hh$ into $Sp_4 (\R)/ U(2)$.     One can do this fixing one Iwasawa decomposition of $SL(3,\mathbb R)$, but one can also do this allowing the map $\iota_1:\hh \rightarrow SL(3, \R)/O(3)$ and the map $\iota_2:SL(3, \R)/O(3) \rightarrow SP(4,\R)/U(2)$ to be constructed using different choices of Iwasawa decompositions for $SL(3,\R)$. The choice of $AN$ in an Iwasawa decomposition depends on a choice of a Weyl chamber at infinity or equivalently of a full flag on $\R^3$.

The nature of this map will depend on how $P_1$, $P_2$, and $l$ are situated relative to the flag $V_1 \subset V_2$.     Consider the ``generic" case, where $l \not\subset V_2$ and $V_1 \not\subset P_1 \cup P_2$.  This corresponds to a ``generic" choice of two Weyl chambers in $SL(3, \R)/O(3)$, i.e. a choice of two opposite chambers.    There are then two special lines, namely $l_1 = P_1 \cap V_2$ and $l_2 = P_2 \cap V_2$.

\begin{itemize} \item The chamber in $\hh$ with endpoints $(l_1,l_2)$ maps to the two chambers corresponding to the flags $l_1
    \subset V_2$ and $l_2 \subset V_2$.    These two are both adjacent to $V_1 \subset V_2$ and so each map to a
    single chamber.    The image under the composition thus is two chambers.

\item The chamber $(l_1,l)$ maps to the single flag $l_1 \subset P_1$.  This is an adjacent chamber and so the
    image
    under the composition is a single chamber.   The same holds for the symmetric case $(l,l_2)$.

\item The chamber $(l,l)$ maps to the two flags $l \subset P_1$ and $l \subset P_2$.  Both of these are opposite,
    and so the image under the composition is four chambers.

\item A chamber $(l_1, l_2')$ where $l_2' \notin \{l,l_2\}$ maps to the two flags $l_1 \subset U$ and $l_2' \subset
    U$, for $U$ the span of $l_1$ and $l_2'$.   The first is adjacent and the second opposite, so the image has
    three chambers.    The same holds for the symmetric cases $(l_1',l_2)$.

\item A chamber $(l, l_2')$ (with $l_2'$ as above) maps to the single flag $l_2' \subset P_2$.   This is opposite
    and so maps to two chambers.

\item There is a bijection between the lines in $P_1 \setminus \{l_1,l\}$ and lines in $P_2 \setminus \{l_2,l\}$
    defined by saying two lines correspond iff their span contains $V_1$.    A chamber with endpoints $(l_1',l_2')$
    maps to the two flags $l_1' \subset U$ and $l_2' \subset U$ where $U$ is the span of $l_1'$ and $l_2'$.
    These
    are both adjacent if $l_1'$ and $l_2'$ correspond under the above bijection, and otherwise are both opposite.
    In the former case the image under the composition is two chambers, and in the latter is four.

\end{itemize}

\noindent {\bf Remark:} the bijection in the last item extends to the full boundaries by sending $l_1 \to l_2$ and $l \to l$.    The graph of this bijection is the boundary of an isometrically embedded $\hyp^2$ in $\hh$, and the composite embedding "bends" interestingly along it.

Which flats in $\hh$ map nicely to  $Sp_4 (\R)/ U(2)$?  If the image is to be a single flat and the map is to be linear, then every chamber must map to two.   A quick enumeration of the possibilities shows this never happens.     The only cases where the image is a single flat (eight total chambers) are:

\begin{itemize} \item The flat $l l_1\times ll_2$ : here the quadrant $(l,l)$ maps to half the flat, the quadrants $(l_1,l)$ and
    $(l,l_2)$ map to single chambers, and the quadrant $(l_1,l_2)$ maps to two chambers.
\item The flats of the form $l_1 l_1' \times l l_2'$ where $l_1'$ and $l_2'$ correspond  ( or the symmetric case $l
    l_1' \times l_2 l_2'$):   The quadrant $(l_1,l_2')$ maps to three chambers, $(l_1,l)$ maps to a single chamber,
    and the other two map to two chambers each.
\end{itemize}

\noindent Thus no flat maps linearly to one flat and the map cannot be an $AN$ map.  It is easy to check that a generic flat maps to 16 chambers and one can enumerate all possible behaviors by continuing the analysis above.  It would be interesting to have a less computational approach to showing that a map is not $AN$.  In \cite{Nguyen}, Nguyen gives some examples quite similiar to those we just discussed, where one can check that a map is not $AN$ by checking that no flat maps to a single flat, but the example above shows that this is not the only obstruction.

\section{Rigidity}\label{rigidity}

If $X_1$ quasi-isometrically embeds in $X_2$ of equal rank, then every maximal flat in $X_1$ gives a quasi-flat in $X_2$.    By results of \cite{KL} and \cite{EF}, such a quasi-flat is at bounded distance from a finite union of flats.    As we want to control the intersection of two quasi-flats we need more information on precisely what a quasi-flat can look like.    For simplicity we describe the results for bilipschitz embeddings of $\R^n$ into Euclidean buildings, since the results for quasi-flats in Eucldiean buildings or symmetric spaces follow formally by passage to asymptotic cones.  For background on buildings, symmetric spaces, asymptotic cones and their relations to one another, we refer the reader to \cite{KL}.

\subsection{Structure of Quasi-Flats in Buildings}

In this subsection,  $X$ will also be a Euclidean building of rank $d$. By a {\bf flat} we mean a top dimensional, isometrically embedded $\R^d$ in $X$.    Such a flat comes with a distinguished family of codimension one affine subspaces, which we call {\em Weyl hyperplanes}.    An affine subspace of a flat which is the intersection of the hyperplanes containing it is called a {\em subflat}.    Each subflat comes with an induced pattern of hyperplanes by intersection, and we continue to call these Weyl hyperplanes (of the subflat).    By a {\em Weyl box} in a (sub)flat we mean a convex set which is a finite intersection of closed Weyl halfspaces.  We now state a more precise form of Lemma \ref{lemma:morse} from the introduction. The description of the exceptional set here takes place in the range rather than the domain, but the two are, of course, related by a bilipschitz map.

\begin{lemma} \label{quasiflats} For any $K$, there is an $n$ so that for any $K$-bilipschitz embedding $f$ of $\R^d$ into $X$ there is a finite set $\{S_1,S_2,\cdots,S_n\}$ of codimension $2$ subflats of $X$ so that every point $x$ in the image of  $f$ not contained in any of the $S_i$ has a neighborhood contained a flat. \end{lemma}

\begin{proof}

We begin by discussing the local structure of a union of two flats.  For this we need some vocabulary.
Given a finite collection of disjoint rays, we form a {\em star} by identifying their initial points. A {\em fan} is a product $F=S \times \mathbb R^k$ where $S$ is a star.  If $S$ consists of exactly $4$ rays, then $F$ is a union of two flats intersecting transversely.

\begin{lemma} There is a $k$ (depending only on $X$) so that for any flats $F$ and $F'$ of $X$, off of at most $k$
codimension two subflats, each point of $F \cup F'$ has a neighborhood which is either contained entirely in a flat
or which is contained in a fan.
\end{lemma}

\begin{proof}
The intersection $F \cap F'$  is a Weyl box in a (sub)flat of each.   If this subflat is codimension two (or more)
then we are done (with $k=1$).   If it is codimension one then the intersection is a box in a hyperplane H.   The interior
points  of the box have neighborhoods which are the contained in the union of two flats which forms a fan.   Thus the points
at which the conclusion of lemma fails are the boundary points of the intersection.    The number of faces of such a
box is bounded by the combinatorics of the Coxeter system of $X$, and each face is contained in a codimension two
subflat of $F$ and $F'$.

Finally, if the intersection is of codimension zero, then the conclusion of the lemma holds in the interior and exterior
of the intersection (these points are locally contained in a single flat) and at points in the interior of faces of
the box of intersection (these are locally in a pair of transverse flats forming a fan).    Thus the points at which  it fails are  contained in the codimension two and higher facets of the boundary of the box.   As before, the number of these  facets is bounded by the Coxeter system, and each facet is contained in a codimension two subflat.
 \end{proof}

 \begin{corollary} For any $j$ there is an $m$ so that for any collection $F_1,F_2,\cdots,F_j$ of flats,  off of at
 most $m$ codimension two subflats (called {\bf exceptional}), each point of $\cup F_i$ has a neighborhood which is
 either contained entirely in a flat or which is contained in the union of finite flats meeting pairwise transversely
 along a common hyperplane $H$.
 \end{corollary}

\begin{proof} Let $x$ be a point of the union $\cup F_i$ off of the exceptional sets produced by the previous  lemma for the pairwise intersections.   If the various hyperplanes through $x$ along which the flats branch are the same in a neighborhood of $x$ then we are done.   If two of the flats intersect in something codimension two or more then we add that subflat to the collection of exceptional subflats (the number of such is bounded  by ${j}\choose{2}$).     Thus we need only address points at which three or more flats intersect.

If three flats intersect pairwise in a subflat of codimension two more then we add this intersection to the set of exceptional subflats.

Finally, if a three flats intersect near $x$ in a box in a  hyperplane $H$ then the lemma holds near $x$ unless unless $x$ is on the boundary of the box, and is thus in one of the codimension two subspaces containing the faces of the box.  The number of these is similarly bounded by $j$ and $X$. \end{proof}

We can now finish the proof of Lemma \ref{quasiflats}.  By \cite{KL}[Corollary 7.2.4], there is an $s$ depending only on $K$ so that image of $f$ is a subset of the union of at most $s$ flats.  Let $Y$ be the union of these flats.  Applying the corollary  gives a family of exceptional subflats off of which every point of $Y$ is has a neighborhood which is either flat (in which case there is nothing to prove) or finite collection of Weyl halfspaces meeting along a common hyperplane $H$.   In the latter case, we locally have an embedding of $\R^d$ into $\R^{d-1} \times S$ where $S$ is finite union of rays meeting at a single point.    As such an embedding must visit only two of the rays in $S$, we have the image contained in a union of two half spaces meeting along a hyperplane.   Such a union is itself a flat in $X$. \end{proof}

\subsection{Flats go to flats}

In this subsection, we prove that under hypotheses analogous to those of Theorem \ref{rigidityThm}, bilipschitz embeddings of Euclidean buildings take flats to flats.

\begin{proposition}\label{flatstoflats}  Let $X$ and  $Y$ be Euclidean buildings of equal rank, such that every $K$-quasi-conformal linear embedding of patterns is conformal.  The every $K$-bilipschitz embedding of $X$ in $Y$ sends flats to flats and is a similarity along every flat. \end{proposition}

\begin{lemma}\label{lemma:hyperplanes} Let $\phi$ be an embedding as above and $F$ a flat in $X$.   For any hyperplane $H$ of $F$, almost every point $x \in H$ has a neighborhood mapped to a hyperplane of $Y$. \end{lemma}

\begin{proof} By Proposition \ref{quasiflats}, we know $\phi(F)$ is locally contained in a flat off of a finite set of codimension 2 subflats.     Let $\Sigma \subset F$ be the pre-image of these exceptional  subflats.    For any $x \in F \setminus \Sigma$, $\phi$ is locally a map into a single flat.    By Rademacher's theorem, $\phi$ is differentiable at almost every point in such a neighborhood, with derivative $K$-quasi-conformal linear map.   Since these neighborhoods cover $F \setminus \Sigma$, this holds at almost every point of $F \setminus \Sigma$.

Let $x$ be such a point of differentiability and $H$ a hyperplane of $F$ through $x$, let $y = \phi(x)$ and  $G$ be a flat in $Y$ which locally contains $\phi(F)$ near $y$.   Choose $F'$ a flat in $X$ which intersects $F$ in $H$.      The image of $F'$ is also a bilipschitz image of a flat, and so has the structure given by Proposition \ref{quasiflats}.    We have $\phi(H) = \phi(F) \cap \phi(F')$.    Off of the exceptional points of $\phi(F')$ this is locally the intersection of two flats which are transverse as $F$ and $F'$ are.    Thus $\phi(H)$ is (locally) contained in the closure of a finite union of hyperplanes in $G$. By the differentiability of $\phi$ at $x$, $\phi(H)$ much locally be equal to a single such hyperplane. \end{proof}

\begin{proof}[Proof of Proposition \ref{flatstoflats}] We now finish the proof of Proposition \ref{flatstoflats}.    As in the proof of the lemma, let $x$ be a point of differentiability of $\phi$ in $F \setminus \Sigma$.   Lemma \ref{lemma:hyperplanes} implies the derivative is a linear embedding of patterns at $x$, and so is 1-quasiconformal.    As this holds almost everywhere, work of Gehring implies $\phi$ is smooth on $F \setminus \Sigma$.      By the assumed rigidity of patterns, the derivative at each point is a scalar multiple of one of a finite number of isometric pattern embeddings, where the scalar is bounded between $\frac{1}{K}$ and $K$.   Since $\Sigma$ cannot disconnect as it is codimension at least two and the map is $C^1$ off of $\Sigma$, we know that off of $\Sigma$ we, in fact, have a scalar multiple of a single isometry.
Comparing the scaling along different lines in the pattern, we see that, possibly after re-scaling the metric on $Y$, that $\phi$ is isometric on $F \setminus \Sigma$.   By continuity, $\phi$ is therefore an isometry on all of $F$.   This forces the image to be a single flat.  A similar use of Gehring's theorem occurs in \cite{D}.  As all of our maps are not just quasi-conformal but bilipschitz, we are really using an easy special case of this theorem.\end{proof}

\subsection{Proof of rigidity results}

 We know prove our main rigidity reslts from the results of the last section.   First, we conclude the analogue of our rigidity results for bilipschitz embeddings of buildings:

\begin{theorem}\label{buildingtheorem} Let $X$ and $Y$ be buildings of equal rank.   If there is a bilipschitz embedding of $X$ in $Y$ then there is a linear embedding of their Weyl chamber patterns.   If all such pattern embeddings which $K$-quasi-conformal are conformal, then all $K$-bilipschitz embeddings of $X$ in $Y$ are, up to rescaling, isometric embedding of subbuildings. \end{theorem}

\begin{proof}
The first conclusion is an immediate consequence of Lemma \ref{lemma:hyperplanes}.  For the second, 
by Proposition \ref{flatstoflats}, each flat in $X$ maps by a similarity to a flat in $Y$.   Since the dilation must be equal on two flats intersecting in anything of positive dimension, and all flats can be connected by such a chain of flats, we can rescale the metric on $Y$ so that flats map isometrically to flats.    Since every geodesic is contained in a flat, the map is globally isometric and totally geodesic. \end{proof}

The proof of our main rigidity result requires one really new ingredient, the main technical result of \cite{KL2}.

\begin{proof}[Proof of Theorem \ref{rigidityThm}]

Given a quasi-isometric embedding of $X_1$ into $X_2$ satisfying the condition of the first bullet point in Theorem \ref{rigidityThm}, by passing to the asymptotic cones, we obtain a bilipschitz embedding of buildings. By Theorem \ref{buildingtheorem} this is impossible as there are no linear embeddings of their Coxeter systems. If $X_1$ and $X_2$ satisfy the conditions of the second bullet point in Theorem \ref{rigidityThm}, then the map on asymptotic cones is an isometric embedding by Theorem \ref{buildingtheorem} and so in particular, every flat is mapped to a single flat. By \cite[Lemma 7.1.1.]{KL}, the fact that flats go to individual flats in all asymptotic cones implies that every flat maps to within bounded distance ( depending only on the quasi-isometry constants ) of a single flat.     Since intersections of flats encode the Weyl chamber patterns, we know that the Weyl hyperplanes map to within uniformly bounded distance of Weyl hyperplanes.   These affine foliations are quite rigid - in particular, by \cite[Lemma 7.1.12]{MSW2}, the quasi-isometry is at bounded distance from an affine map preserving patterns.     Thus every apartment in the Tits boundary of the domain maps to a well defined apartment in the Tits boundary of the domain.  Further, this maps respects the decomposition into chambers.   It is now easy to check that the image is a sub-building: the isometric embedding of patterns induces an inclusion of Weyl groups $W' < W$, and charts from the domain building structure push forward to charts in the range which are $W'$ compatible and therefore $W$ compatible. By \cite[Theorem 3.1]{KL2} it is the boundary of a sub symmetric space (or building) $Y' \subset Y$, and the image of the quasi-isometric embedding is at bounded distance from $Y'$.    So the embedding is a quasi-isometry $X$ to $Y'$, and is therefore at bounded distance from an isometry by the main results of either \cite{EF} or \cite{KL}.  (We quote these last results only for simplicity of exposition.  At this point, we know much more, e.g. that the quasi-isometry from $X$ to $Y'$ is bounded distance from an isometry along each flat, so one can use only a small part of the arguments from those papers.) \end{proof}

\section{Pattern embeddings}\label{patterns}

In this section we do the necessary algebraic analysis to deduce Corollary \ref{corollary:classification} from Theorem \ref{rigidityThm}.  Since it does not require too much more work, we also analyze essentially all settings in which $AN$ maps are possible, at least when the range is classical.

The geometric information needed in our rigidity results (and the algebraic information needed for our constructions of exotic embeddings via $AN-$maps) is closely related to the restricted root systems of  semi-simple groups over $\R$.    Let $X$ be a symmetric space or Euclidean building of rank $r$.   The maximal flats are $r$-dimensional Euclidean vector spaces (canonically identified with $A$).    Each such flat has a finite collection affine foliations, along which the flat meets other maximal flats transversely.   These are precisely the level sets of the roots.

One results about $QI$-embeddings say that if $X_1$ quasi-isometrically embeds in $X_2$ ( of equal rank ) then there is a linear isomorphism $A_1 \to A_2$ which preserves these foliations, and that if we know all such linear maps are conformal then the $QI$-embedding is at bounded distance from isometric.    Using the metric on the flats we can used duality to work in terms of the root vectors rather than the hyperplanes (which are their perpendiculars).  In the following, we study pattern preserving maps of flats using this duality.

\subsection{Pattern Embeddings}

\begin{defn} By a {\bf pattern embedding} of a root system on $R_1$ in a root system on $R_2$ we mean a linear embedding of the underlying vector spaces which sends the root vectors of $R_1$ to scalar multiples of the roots of $R_2$.     We say the pattern embedding is {\bf equal rank} if the root systems are equal rank (and so the embedding is induced by an isomorphism of vector spaces), and that the embedding is {\bf conformal} if the linear map is a conformal embedding. \end{defn}

We want to discuss in more detail when a non-conformal, equal rank pattern embeddings exists of one root system in another.    We analyze the reducible case in terms of irreducible components, so we begin by studying general pattern embeddings in the irreducible case.  Because we are using the dual notion of pattern, our results on pattern embeddings in unequal rank do not say anything about linear embeddings sending affine foliations to affine foliations.

In rank two, which is somewhat special, we are simply discussing patterns of lines in the plane.    The  irreducible patterns are $A_2$, $BC_2$, and $G_2$ - which are, respectively: three, four, and six lines arranged symmetrically, and the reducible pattern $A_1 \times A_1 = D_2$ which is two perpendicular lines.    It is straight forward to check that there are pattern embeddings for any pair where the range has at least as many lines as the domain.       These rank two embeddings will be the building blocks for our study of the higher rank cases.      To this end we note that there are three basic non-conformal embeddings : $A_2$ into  $BC_2$, $D_2$ into $A_2$, and $D_2$ into $BC_2$ (for this last case there are two distinct embeddings, one conformal and one not).     We have excluded the case of target $G_2$ as it does not occur inside any higher rank irreducible root system.

In rank $n > 2$ there are three families of irreducible patterns:  $A_n$, $BC_n$, $D_n$ (this only for $n>3$ as $D_3 = A_3$ and $D_2$ is reducible ), plus the exceptional $E_6$, $E_7$, $E_8$, and $F_4$ (see, for example,  \cite{Knapp}, for a detailed description of these exceptional root systems).     For reference, the classical root systems are:

\begin{enumerate} \item $A_n$ for $n\geq 2$ with $V=\{(x_0,\cdots,x_n) \in \R^{n+1}: x_0 + \cdots + x_n =0\}$ and roots $x_i - x_j$ . \item $B_n$ for $n \geq 2$ with $V = \R^n$ and roots $\{x_i  \pm x_j\} \cup \{\pm x_i\}$. \item $C_n$ for $n \geq 2$ with $V = \R^n$ and roots $\{x_i  \pm x_j\} \cup \{\pm 2x_i\}$. \item $BC_n$ for $n \geq 2$ with $V = \R^n$ and roots $\{x_i  \pm x_j\} \cup \{\pm x_i\} \cup \{\pm 2x_i\}$. \item $D_n$ for $n \geq 4$ with $V = \R^n$ and roots $\{\pm x_i  \pm x_j\} $. \end{enumerate}

While the root systems $B_n$, $C_n$, and $BC_n$ are all distinct, they all have the same root patterns.   We will refer to this pattern, somewhat sloppily, as $BC_n$.

Given a root system on $V$, a two dimensional subspace $U \subset V$ is {\bf special} if its intersection with the roots is irreducible (so $A_2$ or $BC_2$, as $G_2$ does not occur inside any irreducible root systems other than itself).    Two roots whose span is special we will call {\bf related}.   A root system all of whose special subspaces are of type $A_2$ is called simply laced (these are the $A$, $D$, and $E$ families).     As above, simple counting the roots in various special planes shows that a pattern that embeds in a simply laced system is itself simply laced.

We need some facts about special planes in root systems.   We will call a special plane {\bf full} if it is of type $A_2$ in a simply-laced root system or of type $BC_2$ in doubly-laced root system.  (All root systems except $G_2$ which are not simply laced are doubly laced, $G_2$ is triply laced, but we will not need this, or the usual definitions of these concepts, here.)

\begin{lemma} If $R$ is an irreducible root system then \begin{itemize} \item Any two full planes are equivalent under the action of the Weyl group. \item For any two roots $r$ and $r'$ there is a sequence of roots $r=r_0, r_1, \cdots , r_n=r'$ such that the span
    of $r_i$ and $r_{i+1}$ is a full special plane for all $i$.
\end{itemize} \end{lemma}

\begin{proof}

Let $U \subset V$ be a special plane.  Since $R$ is finite, we can choose a linear functional on $V$ which vanishes on $U$ but not on any root outside $U$.   By perturbing this functional we can find a functional which is non-zero on all roots and whose values on roots in $U$ is much smaller than its values on any other roots.     Using this functional to define the positive roots, we see that the  positive roots in $U$ are generated by the positive simple roots in $U$.   Thus $U$ comes from an edge of the corresponding Dynkin diagram.   It is standard that the Weyl group is transitive both on positive systems and on simple edges of the Dynkin diagram  spanned by roots of a fixed length \cite{Knapp}.   Since any diagram contains at most one doubled edge, the first part of the lemma follows.

For the second part, we first consider the weaker equivalence generated by all special planes.    Since two roots which are not orthogonal necessarily span a special plane, the equivalence classes give a decomposition of $V$ into orthogonal factors, with every root contained in a factor.   Thus irreducibility implies there is only one factor.

This completes the second clam for simply laced root systems.  For doubly laced systems we need to see that two roots in a common special plane of type $A_2$ are connected by a chain of full special planes.    The first paragraph above shows that there are at most two Weyl orbits of such planes (one spanned by short roots and one by long roots),  and it is trivial to check the claim for one such plane of each type in both $BC_n$ and $F_4$.

\end{proof}

\begin{lemma} A pattern embedding of irreducible root patterns that maps all special planes to special planes of the same type is conformal. \end{lemma}

\begin{proof}

The starting point is the observation made above that such a map is conformal on all special planes.   Thus the metric is scaled by the same factor along any two related root lines in the domain.   By irreducibility this means the map scales all root lines by the same factor.    The metric on $V$ is determined by its restrictions to the root lines (this is a linear algebra exercise for $A_n$ and $D_n$, and all the other root systems are supersets of these).

\end{proof}

\begin{corollary}  Every pattern automorphism of an irreducible pattern is conformal. \end{corollary}

This can also be shown directly : after composing with an element of the Weyl group, the automorphism preserves the positive cone, and therefor the simple roots.    Since it must also preserve the types of the special plane, this means it corresponds to an automorphism of the Dynkin diagram ( ignoring the direction of multiple edges).    These are isometric automorphisms of the root systems if the directions are preserved ( automatic except for $BC_2$, $G_2$, and $F_4$ which have a conformal pattern automorphism switching the long and short roots ).

\begin{lemma}  Let $P$ and $P'$ be irreducible root systems of equal rank (excluding $G_2$).  If $P$ and $P'$ are both simply laced, or neither is, then every embedding of $P$ in $P'$ is conformal. \end{lemma}

\begin{proof}

If both are simply laced, this is an immediate corollary of the previous lemma. The argument when neither are is the same as there once one notes that for non simply laces root systems, irreducibility implies than any two roots are connected by a chain of special planes of type $BC_2$.

\end{proof}

Conformal embeddings corresponding to inclusions of sub root systems.  Up to symmetries, these are all the obvious inclusions of one root system as a subset of another in the coordinates above, which we call {\bf standard} (\cite{Knapp}).  Each root system includes canonically into the higher rank root systems in the same family, and $A_{n-1} \subset D_n \subset BC_n$.    Here the isomorphism of $A_3$ and $D_3$ causes some confusion, as the two inclusions into $BC_n$ for $n>3$ are distinct.

By the above lemmas, the only  possibilities for  non-conformal pattern embedding above rank two are one of $A$, $D$ or $E$  into one of  $BC$ or $F$.    We will put off a discussion of the exceptional root systems until later, and first address the classical examples:

\begin{lemma}\label{weird} There are, up to the actions of the Weyl groups (and rescaling),  two non-conformal pattern embeddings of $A_n$ into $BC_m$ ( for $m \geq n$) :

\begin{itemize}

\item The roots $x_i - x_j $ for $i>j>0$ map to $y_i - y_j$, and the roots $x_i - x_0$ map to  $y_i$. \item The
    roots $x_i - x_j $ for $i>j>0$ map to $y_i - y_j$, and the roots $x_i - x_0$ map to $y_i + y_1$ (meaning
    the line to $2y_1$ when $i=1$).

\end{itemize} \end{lemma} \begin{proof}

Since the embedding is non-conformal, the prior lemmas imply that there is some special plane in $A_n$ which maps to one of type $BC_2$ in $BC_n$.     The Weyl groups are transitive on special planes in the domain and on type $BC_2$ planes in the range.      Thus we may assume that the plane spanned by $x_1-x_0$ and $x_2-x_1$ maps to the plane spanned by $y_1$ and $y_2$ in the range.      Up to automorphism there are two such maps, depending on whether the missing root line is one of the $y_i$ or $y_1 \pm y_2$.      Thus we may arrange for the map to behave as indicated for the roots $x_i - x_j$ with $ 2 \geq i > j $.

Assume, for induction, that we can apply an automorphism to put the map in one of the two standard forms for all the roots $x_i - x_j$ with $ k \geq i > j $.    The root $x_{k+1} - x_k$ is related to all the roots $x_k - x_i$ for $k > i$, so its image must be related to $y_k - y_i$ for all $ k > i \geq 1$ and to either $y_k$ or $ y_k + y_1$ respectively.  It must also be outside the span of those roots because $x_{k+1} - x_k$ is in the domain.    The only such roots are $\pm y_k \pm y_s $ for $s>k$.     By applying an automorphism of $BC_n$ fixing the first $k$-coordinates, we may assume $s = k+1$ and that $y_{k+1}$ occurs with positive sign.    So $x_{k+1} - x_k$ maps to $\lambda (y_{k+1} \pm y_k)$ with $\lambda >0$.

Since $(x_{k+1} - x_k) + (x_k - x_i) = x_{k+1} - x_i$ is a root in the domain for all $0< i <k $, the image must be a scalar multiple of a root.      By induction and linearity, the image is $\lambda y_{k+1} + (1 \pm \lambda) y_k - y_i$.     Since no roots in the range have three non-zero coefficients and $\lambda >0$, we must have $x_{k+1} - x_k$ mapping to $y_{k+1} - y_k$ as claimed.

\end{proof}

\begin{lemma} Every pattern embedding of $D_n$ (for $n\geq 4$) into $BC_m$ is conformal (and hence standard) \end{lemma} \begin{proof}

As before, a non-conformal embedding must send some special plane to one of type $BC_2$.    The Weyl group of $D_n$ are transitive on special planes, so we may assume one of the planes mapping non-conformally occurs in the canonical $A_{n-1}$ inside $D_n$.     By the remark above, this means the pattern embedding can be taken to be one of the two above maps sending $A_{n-1}$ into $BC_{n-1} \subset BC_n$.    Explicitly, the map does one of the following:

\begin{itemize}

\item The roots $x_i - x_j $ for $i>j>1$ map to $y_i - y_j$, and the roots $x_i - x_1$ map to  $y_i$ (for $i>1$).
    \item The roots $x_i - x_j $ for $i>j>1$ map to $y_i - y_j$, and the roots $x_i - x_1$ map to $y_i + y_2$ for
    $i>1$
    (meaning the line to $2y_2$ when $i=2$).

\end{itemize}

Consider the root $x_1 + x_2$.   It is not in the span of $A_{n-1}$ so its image must have nontrivial $y_1$ component.    It must also be related to all $y_i - y_2$ the images of $x_i - x_2$ ( for all $i>2$).  The only such roots are $\pm y_1 \pm y_2 $.   However, none of these are related to the images of all the  $x_i - x_1$ for $i>2$, so no such pattern embeddings exist.

\end{proof}

For the exceptional root systems the conformal examples are all standard.   Since $B_4$ and $C_4$ are both isometric sub root systems of $F_4$, there are non-conformal embeddings (and even $AN$-maps) from $A_4$ into $F_4$.    Likewise, there is a unique non-conformal embedding of $D_4$ into $F_4$, via a straightforward calculation as above starting with $A_3 \subset D_4$ into $B_3$ or $C_3$ inside $F_4$.    For completeness we give it here on a basis (the rest is determined by linearity):

\begin{itemize}

\item $x_2 - x_1 \mapsto y_1$ \item $x_3 - x_2 \mapsto y_2$ \item $x_4 - x_3 \mapsto \frac{1}{2}(y_4 + y_3 - y_2 -
    y_1)$ \item $x_1 + x_2 \mapsto \frac{1}{2}(y_4 + y_3 + y_2 - y_1)$

\end{itemize}

This gives a complete picture of pattern embeddings between irreducible root systems.     We next turn to the cases where the domain is reducible.    For simplicity, we will only discuss the classical root systems here as the other work out similarly.   Likewise we will only consider the case of the domain as a product of two factors as those with more can all be built by compositions of such.

 Suppose $R = R_1 \oplus R_2$ with each $R_i$ irreducible.   To embed  $R$ into $R'$  of equal rank one needs
 embeddings of both $R_i$ into $R'$ so that the image vector spaces $V_i$ decompose $V'$ as an internal direct sum.
 We will call the embedding {\bf standard} if it is conformal on both of the components  and their images are
 orthogonal.

The earlier lemmas tell us what the image subspaces can look like:

\begin{itemize} \item $A_n$ includes in $D_{n+1}$ naturally as the root vectors with coordinate sum zero.   Composing with signed
    permutations of the coordinates of $D_{n+1}$ one can get an embedding of $A_n$ with image the hyperplane
    $\Sigma
    \varepsilon_i y_i = 0$ for any $\varepsilon_i = \pm 1$.
\item The previous embedding can be composed with the inclusions of $D_{n+1}$ into $D_m$ or $BC_m$ for $m>n$.   The
    image subspace is a hyperplane as above in the subspace spanned by $n+1$ of the coordinates.   These are all
    the
    conformal embeddings of $A_n$ ( for $n>3$ ) into these spaces.
\item $A_n$ includes into $BC_n$ non-conformally (in two ways) which is equal rank.    This can be composed with
    the
    inclusions of $BC_n$ into $BC_m$ for $m>n$.   The image subspaces are the span of $n$ coordinates.   These are
    all of the non-conformal embeddings of $A_n$ into $BC_m$.
\item $D_n$ includes into $BC_n$ conformally and equal rank.    This can be composed with any inclusion $BC_n$ into
    $BC_m$.   The image is the span of $n$ coordinates.  These are all the pattern embeddings $D_n$ into $BC_m$
    (note that for $n=3$ this is also an embedding of $A_3$ into $BC_m$ not covered by the earlier cases ).
\item  For $n \leq m$, $A_n$ includes into $A_m$, $D_n$ into $D_m$, and $BC_n$ into $BC_m$.    These are conformal
    with image the span of $n$-coordinates.
\end{itemize}

The upshot is that the image subspace when the domain is rank $n$ is always either the span of $n$-coordinates or a hyperplane of the specified type in the span of $n+1$-coordinates, and the latter happens only if the domain pattern is $A_n$.

Consider an embedding $R_1 \oplus R_2$ into $R'$ of equal rank.    If both factors have images of the first type ( the span of a subset of coordinates of size equal to the rank of $R_i$ inside $D_m$ or $BC_m$ )   then these subsets of coordinates must be complementary to get $V' = V_1 \oplus V_2$.       This exhausts the possibilities when no factor is of type $A$, and in these cases all the embeddings are conformal.     It also includes all the cases where $A_n$ only embeds non-conformally, and all such embeddings are simply products of embeddings of irreducible composed with one of the conformal embeddings of products of $D$ and $BC$ factors into $D$ or $BC$.

When the target is type $A$, the factors must be so as well.    This case is non-conformal :

\begin{lemma}  There are equal rank, non-standard pattern embeddings of $A_m \oplus A_n $ into $A_{m+n}$. \end{lemma}

\begin{proof}

Divide the $m+n+1$ coordinates for $A_{m+n}$ into two sets, one of size $m+1$ and on of size $n+1$, which overlap in a single element.     The natural inclusions of $A_n$ and $A_m$ as the spans of these are conformal embeddings of $A_n$ and $A_m$.     Their images are disjoint ( as the intersection consists of vectors with coordinate sum zero supported in a single coordinate ) and together span.

\end{proof}

Similarly,

\begin{lemma}  There is a non-standard pattern embedding $A_n \oplus A_m$ into $D_{m+n}$. \end{lemma}

\begin{proof}

The embeddings of the factors are given by $A_k$ includes in $D_{k+1}$ as the vectors of coordinate sum zero, composed with an isometric embedding of $D_{k+1}$ into $D_{m+n}$.      The image of the underlying vector space is therefor a hyperplane in the span of $k+1$ coordinates defined by $ \Sigma \pm y_i = 0$ for some collection of signs.   If the two sets of coordinates overlap in fewer than two coordinates then the image isn't full rank.    If the overlap in more than two then they are not disjoint.      Thus the only embedding comes from overlapping in exactly two coordinates ( say $y_1$ and $y_2$ ) and where the sign choices agree for one of these two and differ for the other.      Explicitly, after signed permutations, it must be:

$$(x_0,  x_1, \cdots, x_n) \times (z_0,z_1,\cdots,z_m)  \mapsto (x_0 - z_0,x_1 + z_1, x_2,  \cdots , x_n, y_2, \cdots, y_m)$$

\end{proof}

In the same way we get  non-standard pattern embeddings $A_n \times D_m$ into $D_{n+m}$ and $A_n \times BC_m$ into $BC_{n+m}$ with all factors conformal.    Up to compositions, we now have a complete list.   To be more precise,  we say that a pattern embedding is {\bf maximal} if it only factors as a composition of two embeddings trivially (meaning that one embedding is actually an isomorphism).   The above constructions together with the list of possible image subspaces for the irreducible embeddings give a complete list.

\begin{corollary} The maximal, equal rank, non-standard pattern embeddings among products of the classical root systems ( with irreducible range) are: \begin{itemize} \item $D_2 = A_1 \times A_1$ into $A_2$ \item $A_n$ into $BC_n$ (in two ways for $n>2$) \item $A_n \times A_m$ into $A_{n+m}$ \item $A_n \times A_m$ into $D_{n+m}$ \item $A_n \times D_m$ into $D_{n+m}$ \item $A_n \times BC_m$ into $BC_{n+m}$ \end{itemize} \end{corollary}

In all of the reducible cases the factors embed conformally but the images are not orthogonal.

\begin{proof}

All listed embeddings have been constructed above.     It remains to show the list is complete.   This is a direct consequence of the classification of irreducible embeddings above and the discussion there of what the image vector spaces can be.     First the case when the image is type $A$:

The factors are all type $A$ and embedded conformally.   Let the target be $A_r$, which we view in the standard way as a subset of vectors in $\R^{r+1}$ with coordinate sum zero.    Each factor has image the vectors lying in the span of a subset of the coordinates.      Two such sets cannot overlap in more than one coordinate or their intersection would contain a root vector.    If two overlap in exactly one factor then we have two factors embedding as in the lemma.   Thus our embedding is a composition of two nontrivial embeddings unless these are the only two factors.       The last possibility is that all factors land in non-overlapping sets of coordinates.    In this case the span of the images is a subspace of $\R^{r+1}$ of co-dimension equal to the number of factors.   Since we are assuming equal rank, this implies there is only one factor and our map is therefore an isomorphism (and, in particular, conformal).

When the target is $D_n$ or $BC_n$ we argue similarly, but things are simpler as the roots span the full $\R^n$ in the stand coordinates for these root spaces.     Again, the image each irreducible factor is either the span of a subset of the coordinates, or a hyperplane in such.     Any two images of the first type must map to disjoint sets of coordinates (as they must be disjoint subspaces).    The map thus factors through the map that combines these factors first, hence is not maximal unless there are only the two factors.     Since the map is non-conformal while these image spaces are orthogonal, at least one of the factors must be $A_k$ embedding non-conformally into $BC_k$, but then again the map is non-maximal.

Thus there is at most one factor of the first type.   Similarly, the argument used above for $A_n$ targets gives that maximality implies there is only one of the second type as well.     These are precisely the cases of one $A_n$ factor, the last three on the list.

Finally, suppose there are no factors of the first type and two of the second.    This is then an embedding of $A_n \times A_m$ into either $D_{n+m}$ or $BC_{n+m}$.    The former is the third listed embedding.     Since the images are assumed of the hyperplane type, the classification of embeddings of $A$ into $BC$ says the image lands in $D \subset BC$, which contradicts maximality.

\end{proof}

\subsection{Relations with $AN$-maps} \label{nonSplitAN}

Recall that if $A \subset G$ is a maximal $\R-$split torus then diagonalizing the action of $A$ on the Lie algebra of $G$ gives the collection of eigenvalue functions $\lambda_i: A \to \R$, which are the restricted roots.    If we choose a hyperplane in $A$  to divide the roots into positive and negative, the $N$ in the $KAN$ decomposition is  the span of the root spaces  $E_{\lambda} \subset G$ with $\lambda$ positive.     Hence, by equivariance,  an $AN$-map from $A_1N_1$ into $A_2N_2$ gives an isomorphism $A_1$ to $A_2$ which sends positive roots to positive roots. In particular, it gives a pattern embedding.

A basic Lie theory fact is that if $r$, $r'$, and $r+r'$  are roots the letting $E_*$ be the corresponding root spaces, one has $[E_r, E_{r'}] = E_{r+r'}$.   The non-trivial part here is that the bracket is surjective.    In the $\R$-split case (or the complex case) this says that any non-zero vectors in $E_r$ and $E_{r'}$ have non-zero bracket.   For an $AN$-map, it follows that the image of the positive roots is {\bf closed}, meaning that if $r$, $r'$ are roots in the domain then $r+r'$ is a root in the range iff it is in the domain.    This puts strong limitations on the kinds of non-conformality that can occur.

To make this precise, and extend to the non-split case, we focus on the two kinds of basic non-conformality : a $D_2$ mapping into an $A_2$ or non-conformally into a $BC_2$,  and an $A_2$ into a $BC_2$.

\begin{lemma}  Let $f$ be an $AN$-map from $A_1N_1$ into $A_2N_2$.    If $r$ and $r'$ are positive  roots in the domain which generate a $D_2$ and which map to non-orthogonal roots $s$ and $s'$ then $s+s'$ is not a root. \end{lemma}

\begin{lemma}  Let $f$ be an $AN$-map from $A_1N_1$ into $A_2N_2$.    Suppose $P$ is a special plane for the domain of type $A_2$ containing positive roots $r$, $r'$, and $r+r'$ and which map to roots $s$, $s'$ and $s+s'$ in a plane of type $BC_2$.    Neither $s+2s'$ nor $s'+2s$ is a root.\end{lemma}

In both cases the conclusion is a restriction which roots in the image plane are positive, constraining them to lie in a single quadrant.     Note that this is the same as saying the image is closed without any assumptions about the sizes of the root spaces (in particular, nothing is assumed $\R$-split).

We need a basic fact from the theory of semi-simple Lie algebras (which is also responsible for the fact quoted above in the split case) : given a root space $E_r$ and $0 \neq v \in E_r$ there is a vector $u \in E_{-r}$ so that $[v,u] = H_r$  where $H_r$ is the element of the maximal torus so that for any root space $E_t$ and any $w \in E_t$,  $[H_r , w]  = (r\cdot t) w$.

\begin{proof}

For the first lemma, let $v$ and $v'$ be  non-zero vectors in the root space $E_s$ and $E_{s'}$ which are in the image of $f$ and let $u$ be as above relative to $v$.    Since $r+r'$ is not a root, the images of $E_r$ and $E_{r'}$ commute, in particular $[v,v']=0$.    Then, by the Jacobi identity:

$$ 0 = [[v,v'],u] = [[v,u],v'] + [v,[v',u]] = [H_s,v'] + [v, [v',u]] = (s \cdot s') v' + [v,[v',u]]$$

By assumption, $s\cdot s' \neq 0$, so $[v,[v',u]] \neq 0$, which implies that $[v',u] \neq 0$ so $ s'-s$ must be a root.   Thus if $s +s'$ is a root, there is a string of roots $(s'-s), (s'-s) + s, (s'-s) + 2s$.   This happens only in $BC_2$ and only if $s$ and $s'$ are orthogonal.

The second lemma is similar.

Let $z$ and $z'$ be elements of the root spaces $E_r$ and $E_{r'}$ with $z''= [z,z'] \in E_{r+r'}$ non-zero ( possible as $[E_r,E_{r'}]=E_{r+r'}$ ).     Let $v'$, $v'$, and $v'' = [v,v']$ be their images in the root spaces $E_{s}$, $E_{s'}$, and $E_{s+s'}$ respectively.

 Suppose $2s+s'$ is a root in the range (the other case is symmetric).    This gives a root string of length three in
 $BC_2$ : $s'$, $s+s'$, $2s+s'$.    As this is the maximal length of a root string, neither $s'-s$ nor $s'+3s$ is a
 root in the range ( and $s$ and $s+s'$ must be orthogonal).    Since $r'+2r$ is not a root in the domain ( the root
 system $A_2$ has only length two root strings), we have $[v'', v]=0$.    Choose a $u \in E_{-s}$ in the range as
 before.     We have, by the Jacobi identity:

$$0 = [ 0,u] = [[v'', v],u] = [[v'',u], v] + [v'',[v,u]] = [[v'',u],v] - ((s+s') \cdot s) v = [[v'',u],v]$$

The Jacobi identity says $0=[v'',u] = [[v,v'],u] = [[v,u].v'] + [v,[v',u]]$.   The second term on the right is zero as $s'-s$ is not a root, and the first is $(s\dot s') v'$.   Thus $s \cdot s' = 0$ which is a contradiction as $ s \cdot (s+s') =0 $.

\end{proof}

We now turn to discussing  the prior pattern embeddings in this context.    Since an $AN$-map must send roots to roots, the two pattern embeddings of $A_n$ into $BC_n$ therefore could come from $AN$-maps into $B_n$ or $C_n$ respectively (or either into $BC_n$ ).

For the first map,  the $A_2$-planes that get sent to $BC_2$ planes are the planes $P_{ij}$ spanned by $x_i-x_0$, $x_j-x_0$, and $x_i - x_j$, which map to $y_i$, $y_j$, and $y_i - y_i$ inside the plane spanned by $y_i$ and $y_j$.  The "missing" roots in the image are $y_i + y_j$ and $-y_i - y_j$.   Thus the lemma above says that $y_i$ and $y_j$ must have opposite signs.  For this to be true for all $i$ and $j$ implies that $n=2$.   Thus this map can arise from an $AN$-map only for $A_2$ into $BC_2$ (in which case this map is the same as the other pattern embedding).

For the second map, the $A_2$ planes mapping to $BC_2$ planes are the planes $P_{i}$ spanned by $x_i-x_0$, $x_i-x_1$, and $x_1 - x_0$, which map to $y_i + y_1$, $y_i - y_1$, and $2y_1$ inside the plane spanned by $y_i$ and $y_1$.    The lemma then says $y_i-y_1$ and $y_i + y_1$ must be the same sign for all $i$.    This happens, for example, if $y_n > \cdots > y_1 >0$.   This map does indeed arise an an $AN$-map (with this order), as described previously.

In the case of the embedding of $D_4$ into $F_4$, it does send roots to roots.    However the image is not closed under addition for any ordering, so there is no $AN$-map.

For the reducible case, an $AN$-map for $A_n \times A_m \to A_{n+m}$ was constructed explicitly already, and the $A_n \times D_m \to D_{n+m}$, $A_n \times BC_m \to BC_{n+m}$, etc. behave similarly.    In all of these cases one can take the Dynkin diagram for the target system and remove one edge so that the two components are the diagrams for the two factors of the domain.    If one orders the roots on the factors so that one is positive and one is negative, additivity is always satisfied.

The one case that does not fit this pattern is $A_n \times A_m \to D_{n+m}$.     If we call the roots in the $A_n$ by $x_i - x_j$ for $n \geq i > j \geq 0$, for $A_m$ by $z_i - z_j$ for $m \geq i > j \geq 0$, and for $D_{n+m}$ by $ y_i \pm y_j $ for $ n+m1 > i > j \geq 0$ then the mapped described sends:

\begin{itemize} \item $x_i - x_j$  to $y_i - y_j$ \item $z_i - z_j$  to $y_{i+n-1} - y_{j+n-1}$ for $i>j>1$ \item $z_i - z_1$ to $y_{i+n-1} - y_1$ for $i>1$ \item $z_i - z_0$ to $y_{i+n-1} + y_0$ for $i>1$ \item $z_1 - z_0$ to $y_1 + y_0$ \end{itemize}

Since the individual factors map conformally, the only conditions from the lemmas are the positive roots from one factor must have non-negative inner-product with those from the other.   From this we conclude:

\begin{itemize} \item $x_i - x_1$ and $z_j - z_1$ have the same sign for all $i$ and $j$ at least $2$ \item $x_i - x_0$ and $z_j - z_0$ have opposite signs for all $i$ and $j$ at least $2$ \item $x_1 - x_0$ has opposite signs to both $z_i - z_0$ and $z_i - z_1$ for $i>1$ \item $z_1 - z_0$ has opposite signs to both $x_i - x_0$ and $x_i - x_1$ for $i>1$ \end{itemize}

This is not possible as $x_2 - x_1$ has the same sign as $z_2 - z_1$, which is opposite to $x_1 - x_0$ and therefore the same as $z_2 - z_0$ hence opposite to $x_2 - x_0$.   So $x_2 - x_1$  and $x_2 - x_0$ have opposite signs, which means it is impossible for both to be opposite to $z_1 - z_0$.

Finally we observe that the closed condition is in general sufficient for the existence of an $AN$-map when the domain is split, so it is exactly these pattern embeddings that arise from $AN$-maps.    Since the subset if closed, it is immediate that the span of the image root spaces gives a subgroup of $N_2$ which is graded exactly as $N_1$ is.      Thus the content of the claim is that this implies the existence of an $AN$-map.    To see this, let $AN$ and $AN'$ correspond to the $\R$-split groups for the given root systems.   Let $\Delta$ be the set of simple roots in the domain.    For each $s \in \Delta$ choose an arbitrary isomorphism $E_s \to E'_s$  between the root spaces in domain and range.      The rest is determined inductively.    Suppose we already have an isomorpshism $E_r \to E'_r$ defined, and $s \in \Delta$ with $r+s$ a root (note that the closed condition says this happens in the domain iff it does in the range, so there is no ambiguity).     Then $[E_r,E_s] = E_{r+s}$ and $[E'_r,E'_s] = E'_{r+s}$, so the map on $E_{r+s}$ is determined.       We need to see it is well-defined.   If $r$ is a positive root then  we need to see that all ways of writing $r=s_1 + s_2 + \cdots s_n$ with all partial sums roots give the same map $E_r \to E'_r$.    To avoid too much Lie theory, we prove this only for domain $A_n$, which is sufficient for all the maps constructed here.

\begin{lemma} \label{typeA}
 Let $s$ and $t$ be simple roots whose sum is not a root.    Suppose $r$ is a positive root such that $r + s$, $r+t$,
 and $r+s+t$ are all roots.     Then the maps $(E_r \otimes E_s) \otimes E_t \to E_{r+s+t}$ and $(E_r \otimes E_t)
 \otimes E_s \to E_{r+s+t}$ are equal. \end{lemma}

\begin{proof}

This is essentially just the Jacobi identity.  Let $u_*$ be vectors in the relevant root spaces.  Then:

$$[u_r, u_s], u_t ] = [[u_r,u_t], u_s] + [u_r,[u_s,u_t]]$$

The second term on the right vanishes as $E_s$ and $E_t$ are assumed to commute.

\end{proof}

Thus we can freely permute adjacent terms provided the simple roots commute.     If we express $x_i - x_j$ as above as a sum of simple roots in $A_n$ (namely $x_{s+1}-x_s$ for varying $s$) then we have the partial sums $x_{i_k} - x_{j_k}$ where at each step either $i_k$ decreases by one or $j_k$ increases by one.    The simple roots of those two types always commute, so the lemma allows us to move any one sum to any other.

\section{Further Questions}\label{Open}

We collect here some further questions concerning quasi-isometric embeddings of symmetric spaces left open by our work and which we believe would be a useful guide to further research.

\begin{itemize}

\item Can one classify the quasi-isometric embeddings of $Sl_3(\R)$ into $Sp_4(\R)$?   It seems unlikely that the
    $AN$-maps are the only ones, but we do not know how to construct others.   Similarly, can one rule out
    quasi-isometric embeddings for spaces of $A_n$ to $C_n$ type where no $AN$-maps exist?   Can one rule out
    embeddings that realize the other linear pattern embedding from lemma \ref{weird} in section \ref{patterns}?
    Might it be that the existence of a quasi-isometric embedding forces (maybe via a limiting argument) the
    existence of an $AN$-map?

\item When do quasi-isometric embeddings exist when rank increases?   As a start, if $X$ quasi-isometrically embeds
    in $Y \times \R^d$ for some $d$, does that imply $X$ quasi-isometrically embeds in $Y$?   More ambitiously, is
    there any sense in which one can describe all the quasi-isometric embeddings when rank increases?  Again, as a
    start, are all quasi-isometric embeddings of $X$ in $X \times \R$ graphs of Lipschitz functions?

\item What can one say about uniformly proper embeddings?   These give perhaps a more natural geometric analogue of
    Margulis' superrigidity.     As a warning, recall that the horospheres give a uniformly proper embedding of
    $\R^n$ in $\hyp^{n+1}$, so rank is not well behaved.

\end{itemize}

\medskip \noindent David Fisher\\ Department of Mathematics\\ Indiana University\\ Bloomington, IN 47401\\ E-mail: fisherdm@indiana.edu

\medskip \noindent Kevin Whyte\\ Department of Mathematics\\ University of Illinois at Chicago\\ Chicago, Il 60607\\ E-mail: kwhyte@math.uic.edu


\begin{thebibliography}{}

\bibitem[BHS]{BHS}
Behrstock, J.;  Hagen, M; Sisto, A Quasiflats in hierarchically hyperbolic spaces.  Preprint arXiv:1704.04271.

\bibitem[BS]{BS} Bonk, M.; Schramm, O. Embeddings of Gromov hyperbolic spaces. {\it Geom. Funct. Anal.} 10 (2000),
    no.
    2, 266Ð306.

\bibitem[BF]{BF}    Brady, Noel; Farb, Benson Filling-invariants at infinity for manifolds of nonpositive curvature.
    {\it Trans. Amer. Math. Soc.} 350 (1998), no. 8, 3393–3405.

\bibitem[EsF]{EF} Eskin, Alex; Farb, Benson Quasi-flats and rigidity in higher rank symmetric spaces. {\it J. Amer.
    Math. Soc.} 10 (1997), no. 3, 653–692.

\bibitem[D]{D} Drutu, Cornelia Quasi-isometric classification of non-uniform lattices in semisimple groups of higher
    rank. {\it Geom. Funct. Anal.} 10 (2000), no. 2, 327–388.

\bibitem[E]{E}  Eskin, Alex Quasi-isometric rigidity of nonuniform lattices in higher rank symmetric spaces. {\it J.
    Amer. Math. Soc.} 11 (1998), no. 2, 321Ð361.

\bibitem[F]{Farb}  Farb, Benson The quasi-isometry classification of lattices in semisimple Lie groups. {\it Math.
    Res. Lett.} 4 (1997), no. 5, 705Ð717.

\bibitem[FN]{FisherNguyen} Fisher, David; Nguyen, Thang Quasi-isometric embeddings of non-uniform lattices, preprint,
    http://front.math.ucdavis.edu/1512.07285.

\bibitem[Fo]{Fo} Foertsch, Thomas Bilipschitz embeddings of negative sectional curvature in products of warped product
    manifolds.  {\it Proc. Amer. Math. Soc.} 130 (2002), no. 7, 2089–2096.

\bibitem[G]{Gromov} Gromov, Mikhael Infinite groups as geometric objects. {\it Proceedings of the International
    Congress of Mathematicians, Vol. 1, 2} (Warsaw, 1983), 385–392, PWN, Warsaw, 1984.
    
\bibitem[H]{H}
Huang, Jingying Top dimensional quasiflats in CAT(0) cube complexes. Preprint arXiv:1410.8195.    

\bibitem[KL1]{KL}  Kleiner, Bruce; Leeb, Bernhard Rigidity of quasi-isometries for symmetric spaces and Euclidean
    buildings. {\it Inst. Hautes \'{E}tudes Sci. Publ. Math.} No. 86 (1997), 115Ð197 (1998)

\bibitem[KL2]{KL2}  Kleiner, Bruce; Leeb, Bernhard Rigidity of invariant convex sets in symmetric spaces. {\it
    Invent.
    Math.} 163 (2006), no. 3, 657Ð676.

\bibitem[Kn]{Knapp} Knapp, Anthony  {\em Lie groups beyond an introduction. Second edition.} Progress in Mathematics,
    140. Birkhäuser Boston, Inc., Boston, MA, 2002. xviii+812 pp. ISBN: 0-8176-4259-5.

\bibitem[L]{L} Leuzinger, E. Corank and asymptotic filling-invariants for symmetric spaces. Geom. Funct. Anal. 10
    (2000), no. 4, 863–873.

\bibitem[LMR]{LMR} Lubotzky, Alexander; Mozes, Shahar; Raghunathan, M. S. The word and Riemannian metrics on lattices
    of semisimple groups. {\it Inst. Hautes \'{E}tudes Sci. Publ. Math.} No. 91 (2000), 5–53 (2001).

\bibitem[MSW]{MSW2}  Mosher, Lee; Sageev, Michah; Whyte, Kevin Quasi-actions on trees II: Finite depth Bass-Serre
    trees. {\it Mem. Amer. Math. Soc.} 214 (2011), no. 1008.

\bibitem[N]{Nguyen} Nguyen, Thang Quasi-isometric embeddings of symmetric spaces and lattice: the reducible setting,
    in progress.

\bibitem[P]{P} Pansu, Pierre Métriques de Carnot-Carathéodory et quasiisométries des espaces symétriques de rang un.
    (French) [Carnot-Carathéodory metrics and quasi-isometries of rank-one symmetric spaces] Ann. of Math. (2) 129
    (1989), no. 1, 1–60.

\end{thebibliography}
\end{document}